\documentclass[revision]{FPSAC2025}
\usepackage{tikz-cd}
\usepackage{bm, relsize}
\usepackage{comment}
\newtheorem{theorem}{Theorem}

\newtheorem{defn}{Definition}

\newtheorem{proposition}{Proposition}

\newtheorem{corollary}{Corollary}
\newtheorem{example}{Example}

\newtheorem{remark}{Remark}

\newcommand{\al}{\alpha}

\newcommand{\NN}{\mathbb{N}} % nonnegative integers
\newcommand{\id}{\operatorname{id}} % identity map
\newcommand{\kk}{\mathbf{k}} % our base ring
\newcommand{\VV}{\mathbf{V}} % Dynkin operator
\newcommand{\VVq}{\mathbf{V}^{(q)}} % Dynkin operator
\newcommand{\BB}{\mathbf{B}} % B-basis
\newcommand{\DD}{\mathbf{D}} % D-basis
\newcommand{\comp}{\operatorname{comp}}

\newcommand{\Des}{\operatorname{Des}}

\newcommand{\col}{\operatorname{col}}
\newcommand{\Comp}{\operatorname{Comp}}

\newcommand{\Set}{\operatorname{Set}}

\makeatletter
\providecommand*{\shuffle}{%
  \mathbin{\mathpalette\shuffle@{}}%
}
\newcommand*{\shuffle@}[2]{%
  \sbox0{$#1\vcenter{}$}%
  \kern .15\ht0 % side bearing
  \rlap{\vrule height .25\ht0 depth 0pt width 2.5\ht0}%
  \raise.1\ht0\hbox to 2.5\ht0{%
    \vrule height 1.75\ht0 depth -.1\ht0 width .17\ht0 %
    \hfill
    \vrule height 1.75\ht0 depth -.1\ht0 width .17\ht0 %
    \hfill
    \vrule height 1.75\ht0 depth -.1\ht0 width .17\ht0 %
  }%
  \kern .15\ht0 % side bearing
}
\makeatother

\newenvironment{verlong}{}{}
\newenvironment{vershort}{}{}

\excludecomment{verlong}
\includecomment{vershort}
\excludecomment{noncompile}

\title{A permutation based approach to the $q$-deformation of the Dynkin Operator}

\author[D. Grinberg \and E.A. Vassilieva]{Darij Grinberg\thanks{\href{mailto:darijgrinberg@gmail.com}{darijgrinberg@gmail.com}}\addressmark{1}, \and Ekaterina A. Vassilieva\thanks{\href{mailto:katya@lix.polytechnique.fr}{katya@lix.polytechnique.fr}}\addressmark{2}}

\address{\addressmark{1}Department of Mathematics, Drexel University, Philadelphia, PA 19104, USA \\ \addressmark{2}LIX, Ecole Polytechnique, Palaiseau, France}

\received{2025-04-13}

\revised{2026-06-05}

\abstract{Introduced by Solomon, the descent algebra is a significant subalgebra of the group algebra of the symmetric group $\mathbf{k}S_n$ related to many important algebraic and combinatorial topics. It contains all the classical Lie idempotents of $\mathbf{k}S_n$, in particular the Dynkin operator, a fundamental tool for studying the free Lie algebra. We look at a $q$-deformation of the Dynkin operator and study its action over the descent algebra with classical combinatorial tools like Solomon's Mackey formula. This leads to elementary proofs that the operator is indeed an idempotent for $q=1$ as well as interesting formulas and algebraic structures especially when $q$ is a root of unity.}

\resume{Introduite par Solomon, l'algèbre de descentes est une sous-algèbre saillante de l'algèbre de groupe du groupe symétrique $\mathbf{k}S_n$ liée à de nombreux sujets importants en algèbre et en combinatoire. Elle contient tous les idempotents de Lie classiques de $\mathbf{k}S_n$, en particulier l'opérateur de Dynkin, un outil fondamental pour étudier l'algèbre de Lie libre. Nous étudions l'action d'une $q$-déformation de l'opérateur de Dynkin sur l'algèbre de descente avec des outils combinatoires classiques tels que la formule Mackey de Solomon. Nous obtenons des preuves élémentaires que l'opérateur est bien un idempotent pour $q=1$ ainsi que de formules et structures algébriques intéressantes en particulier lorsque $q$ est une racine de l'unité.}

\usepackage[backend=bibtex]{biblatex}
\begin{document}

\maketitle
%% note that you DO NOT have to put your abstract here -- it is generated by \maketitle and the \abstract and \resume commands above

%%%%%%%%%%%%%%%%%%%%%%%%%%%%%%%%%%
%%%%%%%%%%%%%%%%%%%%%%%%%%%%%%%%%%
%%%%%%%%%%%%%%%%%%%%%%%%%%%%%%%%%%
\section{Introduction}
%%%%%%%%%%%%%%%%%%%%%%%%%%%%%%%%%%
%%%%%%%%%%%%%%%%%%%%%%%%%%%%%%%%%%
%%%%%%%%%%%%%%%%%%%%%%%%%%%%%%%%%%
Given a base ring $\mathbf{k}$ and a non-negative integer $n \in \mathbb{N}$, denote $[n]=\{1, 2, \dots, n\}$,  $S_n$ the symmetric group on $[n]$ and let $\mathbf{k}S_n$ be the group algebra of the symmetric group.
The \emph{descent algebra} $\Sigma_n$, introduced by Solomon in \cite{Sol76}, is the subalgebra spanned by the elements of  $\mathbf{k}S_n$ such that permutations with the same descent set have the same coefficients.
This algebra $\Sigma_n$ has two elementary bases indexed by subsets of $[n-1]$ or, equivalently, by integer compositions of $n$: the descent class $\mathbf{D}$-basis and the subset $\mathbf{B}$-basis.
%and the idempotent $\mathbf{I}$-basis.
The structure constants of the $\mathbf{B}$-basis admit an elegant combinatorial interpretation %in terms of number of matrices with non-negative integer coefficients with fixed row-sum, column-sum and reading word:
known as Solomon's Mackey formula.
The Dynkin operator $\mathbf{V}_n$ is a remarkable element of $\Sigma_n$ defined as an alternating sum of some descent classes (i.e. elements of the $\mathbf{D}$-basis) playing a fundamental rôle in the study of the free Lie algebra \cite[(8.4.2)]{Reutenauer1993}.
The result that $\mathbf{V}_n$ is indeed (quasi)idempotent follows easily from the fact that $\mathbf{V}_n \mathbf{B}_I = 0$ for all non-empty subsets $I \subseteq [n-1]$.
The latter fact appears in the literature with several proofs, but we provide the first combinatorial proof that derives it from Solomon's Mackey formula.
We also study the $q$-deformation of $\mathbf{V}_n$ introduced in \cite{CalderbankHanlonSundaram94} and compute its left action on the $\mathbf{B}$-basis
as well as the dimension of the right ideal it generates depending on $q$.
% We characterise some interesting algebraic structures in particular when $q$ is a root of unity.
% After an introduction of the main notations and definitions used throughout the paper we state and prove our main results and their consequences. 
%%%%%%%%%%%%%%%%%%%%%%%%%%%%%%%%%%
\subsection{Permutations, integer compositions and subsets}
%%%%%%%%%%%%%%%%%%%%%%%%%%%%%%%%%%
A \emph{composition} $\al = (\al_1, \al_2, \dots, \al_p)$ of a non-negative integer $n$, denoted $\al \vDash n$, is a finite sequence of $\ell(\al) := p$ positive integers such that $|\al| := \sum_i \al_i = n$. Denote $\Comp_n$ the set of compositions of $n$. There is a natural bijection between compositions of $n$ and subsets of $[n-1]$.
Namely, for a composition $\al = (\al_1, \al_2, \dots, \al_p) \in \Comp_n$, let $\Set(\al)$ be the subset of $[n-1]$ defined as $\Set(\al) = \{\al_1,\al_1+\al_2,\dots,\al_1+\al_2+\dots+\al_{p-1}\}$. Conversely, given a subset $I = \{i_1< i_2< \dots < i_p\} \subseteq [n-1]$ let $\comp(I)$ be the composition $(i_1, i_2-i_1,i_3 - i_2, \dots, n-i_p)$ of $n$.
Thus, $|\Comp_n| = 2^{n-1}$ when $n \geq 1$.
\begin{defn}[composition refinement]
Let $\alpha = (\alpha_1, \alpha_2, \dots, \alpha_m)$ and $\beta = (\beta_1, \beta_2, \dots, \beta_p)$ be two compositions of the same integer $n$. We say that $\alpha$ \emph{refines} $\beta$, and write $\alpha \preceq \beta$, if and only if $\alpha$ can be split into $p$ subsequences $\alpha^i = (\alpha^i_1, \alpha^i_2, \dots, \alpha^i_{m_i})$ for $1\leq i\leq p$ such that $\alpha^i \vDash \beta_i$.
Note that the sequence of compositions $(\alpha^i)_{i=1\dots p}$ is uniquely determined by $\alpha$ and $\beta$.
Denote by $\alpha|\beta$ this sequence, and by $\overline{\alpha|\beta}$ the list of the last (rightmost) elements of each $\alpha^i$, that is,
\[
\overline{\alpha|\beta} = (\alpha^1_{m_1}, \alpha^2_{m_2}, \dots, \alpha^p_{m_p}).
\]
\end{defn}
\noindent If $I$ and $J$ are two subsets of $[n-1]$, then $J \subseteq I$ if and only if $\comp(I) \preceq \comp(J)$. We write $I|J$ (resp. $\overline{ I|J}$) instead of $\comp(I)|{\comp(J)}$ (resp. $\overline{\comp(I)|{\comp(J)}}$) for brevity's sake. If $J \nsubseteq I$, we set $\overline{ I|J} = \emptyset$. Finally, if $U$ is a sequence of numbers, we let $|U|$ (resp. $\Sigma U$) denote the number (resp. the sum) of its elements, i.e., we set $\Sigma U := \sum\limits_{u \in U} u$.

%%%%%%%%%%%%%%%%%%%%%%%%%%%%%%%%%%
\subsection{The descent algebra of the symmetric group}
%%%%%%%%%%%%%%%%%%%%%%%%%%%%%%%%%%
Given a permutation $\pi \in S_n$, we consider its \emph{descent set} $\Des(\pi)$ defined as $\Des(\pi) = \{1\leq i\leq n-1| \pi(i)>\pi(i+1)\}.$ 
For $I \subseteq [n-1]$, let $\mathbf{D}_I$ and $\mathbf{B}_I$ be the elements of the group algebra $\mathbf{k}S_n$ defined as \begin{equation*}
\mathbf{D}_I = \sum_{\pi \in S_n; \; \Des(\pi) = I}\pi, \;\;\;\;\;\;\;\; \mathbf{B}_I = \sum_{\pi \in S_n; \; \Des(\pi) \subseteq I}\pi.
\end{equation*} 
As shown by Solomon in \cite{Sol76} in the more general context of finite Coxeter groups, the $\mathbf{D}_I$'s (resp. $\mathbf{B}_I$'s) span the \emph{descent algebra} $\Sigma_n$, a subalgebra of $\mathbf{k}S_n$ of dimension $2^{n-1}$. The families $(\mathbf{D}_I)_{I\subseteq [n-1]}$ and $(\mathbf{B}_I)_{I\subseteq [n-1]}$ are bases of $\Sigma_n$, known as the \emph{descent class} or simply $\mathbf{D}$-basis, resp. the \emph{subset} or $\mathbf{B}$-basis.
The two bases are obviously related through
\begin{equation}
\label{equation.bidj}
\mathbf{B}_I = \sum_{J \subseteq I} \mathbf{D}_J, \;\;\;\;\;\;\;\; \mathbf{D}_I = \sum_{J \subseteq I}(-1)^{|I\setminus J|}\mathbf{B}_J.
\end{equation} 
The structure constants of the $\mathbf{B}$-basis have a combinatorial description %in terms of integer matrices
known as \emph{Solomon's Mackey formula} \cite[Theorem 2]{Willigenburg98}, \cite[Prop. 4.3]{BleLau92} (note the different conventions for the order of multiplication in $S_n$).
To state it, for any $f \in \Sigma_n$ and $K \subseteq [n-1]$, let $[\mathbf{B}_K]f$ denote the coefficient of $\mathbf{B}_K$ in the expansion of $f$ in the $\mathbf{B}$-basis of $\Sigma_n$.
We abbreviate $[\mathbf{B}_K](fg)$ as $[\mathbf{B}_K]fg$.
Furthermore, if $M$ is a $p\times q$ matrix, its \emph{row sums vector} is the $p$-tuple of row sums, its \emph{column sums vector} is the $q$-tuple of column sums, and its \emph{reduced reading word} is the concatenation of the rows of $M$ without zeroes.
% In other words, it is the list of all nonzero entries of $M$ obtained by traversing the cells in lexicographic order.

\begin{proposition}[Solomon's Mackey formula]
\label{prop.solmac}
Let $I, J, K \subseteq [n-1]$.
Let $\mathbb{N}^{I, J}_K$ be the set of nonnegative integer matrices with row sums vector $\comp(I)$, column sums vector $\comp(J)$ and reduced reading word $\comp(K)$.
Then, $[\mathbf{B}_K]\mathbf{B}_J\mathbf{B}_I = |\mathbb{N}^{I, J}_K|$.
\end{proposition}
\noindent Note that $\mathbb{N}^{I, J}_K$ is empty unless $I \subseteq K$. Thus, $[\mathbf{B}_K]\mathbf{B}_J\mathbf{B}_I = 0$ unless $I \subseteq K$. 
The non-zero entries of the $i$-th row in the matrix must give the $i$-th subsequence in $K|I$. Furthermore the sequence of the rightmost non-zero elements of each row in the matrix is exactly $\overline{ K|I}$.
\begin{example}
Assume $n=5$, $I=\{3\}$, $J=K=\{1,3\}$. We have $\comp(I)=(3,2)$, $\comp(J) = \comp(K) = (1,2,2)$. We notice that $I\subseteq K$, thus $\comp(K) \preceq \comp(I)$. Finally $K|I=\left((1,2),(2)\right)$,  $\overline{ K|I} = (2,2)$ and $|\mathbb{N}^{I, J}_K| = 2$ as there are exactly $2$ suitable matrices, namely:
\[
\begin{pmatrix}
1&0  &2  \\
0&2  &0
\end{pmatrix}  \text{ and }  
\begin{pmatrix}
1&2  &0 \\
0&0  &2
\end{pmatrix} .
\]

\end{example}

%%%%%%%%%%%%%%%%%%%%%%%%%%%%%%%%%%
\subsection{The Dynkin operator}
%%%%%%%%%%%%%%%%%%%%%%%%%%%%%%%%%%
A permutation $w\in S_{n}$ is said to be a
\emph{V-permutation} if there exists some $k\in\left[  n\right]  $ such that%
\[
w\left(  1\right)  >w\left(  2\right)  >\cdots>w\left(  k\right) <w\left(
k+1\right)  <w\left(  k+2\right)  <\cdots<w\left(  n\right).
\]
In this case, automatically $w\left(k\right) = 1$.
\begin{defn}[Dynkin operator]
The \emph{Dynkin operator} is the element
\begin{align*}
\mathbf{V}_{n}  &  =\sum_{\substack{w\in S_{n}\text{ is a}%
\\\text{V-permutation}}}
\left(  -1\right)  ^{w^{-1}\left(  1\right)  -1}w
\in \kk S_n.
\end{align*}
\end{defn}

\noindent
The Dynkin operator can also be written as
\begin{align}
\label{eq.Vn=cycs}
\mathbf{V}_{n}
&  =\left(  1-\operatorname*{cyc}\nolimits_{2,1}\right)  \left(
1-\operatorname*{cyc}\nolimits_{3,2,1}\right)  \left(  1-\operatorname*{cyc}%
\nolimits_{4,3,2,1}\right)  \cdots\left(  1-\operatorname*{cyc}%
\nolimits_{n,n-1,\ldots,1}\right),
\end{align}
where $\operatorname*{cyc}\nolimits_{i,i-1,\dots,2,1}$ is the cycle permutation that maps $i,i-1,\ldots,2,1$ to $i-1,i-2,\ldots,1,i$, respectively, and fixes all the other elements of $[n]$.
See \cite[Lemma 1.1]{Garsia1990} for a proof. The Dynkin operator satisfies
\begin{equation}
\label{equation.vn2vn}
\mathbf{V}_{n}^{2}=n\mathbf{V}_{n}.
 \end{equation}
This is proved in \cite[\S 1]{BergeronBergeronGarsia} and \cite[Remark 2.1]{Garsia1990} using the free Lie algebra, in \cite[proof of (5)]{BleLau92} using custom-built combinatorics, and in \cite[Subsection 4.5.2]{sga} using Hopf algebras.
The former two proofs rely on the fact that
\begin{equation}
\label{equation.vndi}
\mathbf{V}_{n}\mathbf{D}_{I}=\left(  -1\right)  ^{\left\vert I\right\vert}\mathbf{V}_{n}
\qquad \text{ for any } I \subseteq \left[n-1\right].
\end{equation}
Using Equation (\ref{equation.bidj}), this can be rewritten as
\begin{equation}
\label{equation.vnbi}
\mathbf{V}_{n}\mathbf{B}_{I}=0
\qquad \text{ for any nonempty } I \subseteq \left[n-1\right].
\end{equation}
This surprising connection with the $\mathbf{B}$-basis
% in Equation (\ref{equation.vnbi})
suggests a link to Solomon's Mackey formula.
Furthermore we felt that it is worth studying how the action of  $\mathbf{V}_{n}$ on the $\mathbf{B}$-basis changes (e.g. in terms of rank and kernel) when the coefficient $-1$ in $(-1)^{w^{-1}\left(  1\right)}$ is replaced by an arbitrary element $-q$ of the base ring $\mathbf{k}$, especially a root of unity. We thus look at the following $q$-deformation of the Dynkin operator. 
\begin{defn}[$q$-deformation of the Dynkin operator, \cite{CalderbankHanlonSundaram94}]
Fix $q$ in the base ring $\mathbf{k}$ and a non-negative integer $n$. The $q$-\emph{deformation} of the Dynkin operator is defined as
\begin{align}
\label{equation.vnq}
\mathbf{V}_{n}^{(q)}  &  :=\sum_{\substack{w\in S_{n}\text{ is
a}\\\text{V-permutation}}}(-q)^{w^{-1}\left(  1\right)  -1}w \in \kk S_n.
\end{align}
\end{defn}
Generalizing \eqref{eq.Vn=cycs}, $\mathbf{V}_{n}^{(q)} =\left(  1-q\operatorname*{cyc}\nolimits_{2,1}\right)  \left(1-q\operatorname*{cyc}\nolimits_{3,2,1}\right)  \cdots\left(  1-q\operatorname*{cyc}\nolimits_{n,n-1,\ldots,1}\right).$
%it is not hard to show that
%\begin{align}
%\label{eq.Vnq=cycs}
%\mathbf{V}_{n}^{(q)}
%&  =\left(  1-q\operatorname*{cyc}\nolimits_{2,1}\right)  \left(
%1-q\operatorname*{cyc}\nolimits_{3,2,1}\right)  \left(  1-q\operatorname*{cyc}%
%\nolimits_{4,3,2,1}\right)  \cdots\left(  1-q\operatorname*{cyc}%
%\nolimits_{n,n-1,\ldots,1}\right).
%\end{align}

In the work of Calderbank, Hanlon and Sundaram \cite{CalderbankHanlonSundaram94}, $\VV_n^{(q)}$ appears  under the name $\eta_n(\alpha)$ for $\alpha = q$ in the context of higher Lie modules: The left ideal of $\kk S_n$ generated by $\VVq_n$ is an $S_n$-representation $V_n(\alpha)$, which is (isomorphic to) the multilinear component of the ``$\alpha$-deformed free Lie algebra'' on $n$ generators (which generalises both the free Lie algebra and the free Lie superalgebra). Our focus is to study $\VVq_n$ using Solomon's Mackey formula and elementary permutation combinatorics. 
\noindent Three specializations of $\VV^{(q)}_n$ are worthy of note. As per the definition $\VV^{(1)}_n = \VV_n$ is the standard Dynkin operator.
Next, $\VV^{(0)}_n = \id \in \kk S_n$. Finally, $\VV^{(-1)}_n$ is the sum of all V-permutations, i.e. the sum of all permutations with empty \emph{peak set}. $\VV^{(-1)}_n$ is the main building block of the peak algebra analog to the Dynkin operator introduced in \cite{Sch05}.

%\begin{verlong}
In the next sections, we proceed with a general expansion of the action of $\mathbf{V}_{n}^{(q)}$ on the $\mathbf{B}$-basis using Solomon's Mackey formula.
Then we provide some consequences:
new proofs of \eqref{equation.vnbi}, \eqref{equation.vndi} and \eqref{equation.vn2vn};
a description of the eigenvalues of the action of $\VVq_n$ on $\Sigma_n$;
and a combinatorial formula for the dimension of its image $\VVq_n \Sigma_n$.
%\end{verlong}
%The main case considered there was when $q$ is a root of unity, since otherwise $\VVq_n$ is simply the whole regular representation $\kk S_n$.
%As our focus is on the element $\VVq_n$ itself rather than just the left ideal it generates, we are interested in all values of $q$.

%%%%%%%%%%%%%%%%%%%%%%%%%%%%%%%%%%
%%%%%%%%%%%%%%%%%%%%%%%%%%%%%%%%%%
%%%%%%%%%%%%%%%%%%%%%%%%%%%%%%%%%%
\section{Action of $\mathbf{V}_n^{(q)}$ on the $\mathbf{B}$-basis of the descent algebra}
%%%%%%%%%%%%%%%%%%%%%%%%%%%%%%%%%%
%%%%%%%%%%%%%%%%%%%%%%%%%%%%%%%%%%
%%%%%%%%%%%%%%%%%%%%%%%%%%%%%%%%%%

Given an integer $n \in \NN$, a parameter $q \in \kk$ and a subset $I \subseteq [n-1]$, we look at the expansion of $\mathbf{V}_n^{(q)} \mathbf{B}_I$ in the $\mathbf{B}$-basis of the descent algebra $\Sigma_n$. We first decompose $\mathbf{V}_n^{(q)}$  in the $\mathbf{D}$-basis and solve the problem for these basis elements.
%%%%%%%%%%%%%%%%%%%%%%%%%%%%%%%%%%
\subsection{Decomposition according to descent classes}
% 
%As a direct consequence to Equation (\ref{equation.vnq}), we get the decomposition of the
First, we decompose the $q$-deformed Dynkin operator in the $\mathbf{D}$-basis of $\Sigma_n$.
\begin{proposition}
\label{proposition.vnqd}
For any $n \in \NN$ and $q \in \kk$, the $q$-deformed Dynkin operator $\mathbf{V}_{n}^{(q)}$ expands as
\begin{equation}
\label{equation.vnqd}
\mathbf{V}_{n}^{(q)} = \sum_{k=1}^n (-q)^{k-1} \mathbf{D}_{[k-1]}.
\end{equation}
As a result, $\mathbf{V}_{n}^{(q)} \in \Sigma_n$.
\end{proposition}
\begin{proof}
For each $k \in [n]$, we know that $\mathbf{D}_{\left[  k-1\right]  }$ is the sum of all permutations $w \in S_n$ with descent set $[k-1]$. But these are precisely the V-permutations $w \in S_n$ with $w^{-1}\left(1\right) = k$.
% \begin{equation*}
% \mathbf{D}_{\left[  k-1\right]  }=\sum_{\substack{w\in
% S_{n}\text{ is a}\\\text{V-permutation;}\\w^{-1}\left(  1\right)  =k}}w.
% \end{equation*}
% Equation (\ref{equation.vnqd}) follows immediately. 
Hence, the right hand side of \eqref{equation.vnqd} agrees with the right hand side of \eqref{equation.vnq}.
The equality \eqref{equation.vnqd} follows, and thus the proposition holds.
\end{proof}
%
%\begin{defn}[$q$-deformed Dynkin expansion in the $\mathbf{D}$ basis]
%For integer $n$ and parameter $q$, decompose the $q$-deformed Dynkin idempotent $\mathbf{V}_{n}^{q }$ as
%\begin{equation}
%\mathbf{V}_{n}^{(q)} = \sum_{k=1}^n (-q)^{k-1} \mathbf{V}_{n,k}
%\end{equation}
%where the summands $\mathbf{V}_{n,k}$ are defined as 
%\begin{equation}
%\mathbf{V}_{n,k}:=\mathbf{D}_{\left[  k-1\right]  }=\sum_{\substack{w\in
%S_{n}\text{ is a}\\\text{V-permutation;}\\w^{-1}\left(  1\right)  =k}}w.
%\end{equation}
%\end{defn}
%

\noindent Define $\mathbf{V}_{n,k}:=\mathbf{D}_{\left[  k-1\right]  }$ for each $k \in \left[n\right]$. We expand $\mathbf{V}_{n,k}\mathbf{B}_I$ for $I\subseteq [n-1]$ in the  $\mathbf{B}$-basis.
\begin{theorem}
\label{thm.vnk}
Given two positive integers $k\leq n$ and two sets $I,K \subseteq [n-1]$ with $I \subseteq K$.
Then, the $\mathbf{B}_K$-coefficient in the expansion in the $\mathbf{B}$-basis of $\mathbf{V}_{n,k}\mathbf{B}_I$ is 
\begin{equation}
[\mathbf{B}_K]\mathbf{V}_{n,k}\mathbf{B}_I = (-1)^{k + |K|}\sum_{\substack{U \subseteq \overline{ K|I};\\ \Sigma U > n-k} }(-1)^{|U|}.
\end{equation}
(The sum runs over subsequences of $\overline{ K|I}$, parameterized by their positions in $\overline{ K|I}$; thus, equal subsequences at different positions lead to repeated addends.)
\end{theorem}

\begin{proof}
We define
\[
\mathbb{N}^{I}_K := \bigcup_{J \subseteq [n-1]}\mathbb{N}^{I, J}_K;
\]
this is the set of all nonnegative integer matrices with row sums vector $\comp(I)$ and reduced reading word $\comp(K)$ that have no zero columns. Given $A \in \mathbb{N}^{I}_K$, we let $\col(A)$ be the number of columns of $A$ and $A^\star$ the sequence of entries in the rightmost column of $A$. We have
$
\mathbf{V}_{n,k} = \mathbf{D}_{\left[  k-1\right]  }
= \sum\limits_{J\subseteq\left[
k-1\right]  }\left(  -1\right)  ^{\left(  k-1\right)  -\left\vert J\right\vert}\mathbf{B}_{J}
$
by Equation \eqref{equation.bidj}.
Hence,
\begin{align}
\nonumber [\mathbf{B}_K]\mathbf{V}_{n,k}\mathbf{B}_{I}  & =\sum_{J\subseteq\left[  k-1\right]
}\left(  -1\right)  ^{\left(  k-1\right)  -\left\vert J\right\vert }%
[\mathbf{B}_K]\mathbf{B}_{J}\mathbf{B}_{I}\\
\nonumber & = \sum_{\substack{J\subseteq\left[  k-1\right]; \\ A\in
\mathbb{N}^{I, J}_K }}\left(
-1\right)  ^{  \left(  k-1\right)  -\left\vert J\right\vert}\ \ \ \ \ \ \ \ \ \ \left(
\text{by Proposition \ref{prop.solmac}}\right)  \\
\label{equation.sum_A}
& = \sum_{\substack{A\in \mathbb{N}^{I}_K; \\ \Sigma A^\star >n-k}}
\left(  -1\right)  ^{ k  - \col(A)}
\end{align}
(since $\col(A) = |J| + 1$ and $\sum A^\star = n - \max J$ for any $A \in \NN^{I,J}_K$).

Now, we shall establish a sign-reversing involution on the set $\NN^{I}_K$ that makes most addends on the right hand side of \eqref{equation.sum_A} cancel.
Namely, call a column of a matrix $A \in \NN^{I}_K$ \emph{splittable}
if it is not the last column and has at least two nonzero entries.
A splittable column $c$ can be split into two by creating a new column
$c'$ immediately to its left and moving the top nonzero entry of $c$
to $c'$.
(All other cells of $c'$ are filled with $0$, and so is the cell of
$c$ from which an entry was moved.)
Conversely, two adjacent columns $c'$ and $c$ of $A$
(with $c'$ left of $c$) will be called \emph{mergeable} if
$c$ is not the last column of $A$ and the column $c'$ has
only one nonzero entry and this entry lies further up than all nonzero
entries of $c$. A mergeable pair of columns $c', c$ can be merged
into one simply by adding them together.
Splitting and merging are mutually inverse operations.
For example:
\begin{align*}
\begin{pmatrix}
1 & 2 & 0 & 1 \\ 0 & 1 & 3 & 2
\end{pmatrix}
\overset{\text{split column } 2}{\longrightarrow}
\begin{pmatrix}
1 & 2 & 0 & 0 & 1 \\ 0 & 0 & 1 & 3 & 2
\end{pmatrix}
\overset{\text{merge columns } 2, 3}{\longrightarrow}
\begin{pmatrix}
1 & 2 & 0 & 1 \\ 0 & 1 & 3 & 2
\end{pmatrix}
\end{align*}
(note that the last column is not splittable by definition).

Thus, all matrices $A \in \NN^{I}_K$ that have a splittable column
or a mergeable pair of columns cancel out from \eqref{equation.sum_A}
(just pick the leftmost splittable-or-mergeable column, and
merge or split it, as in \cite[proof of Theorem 2.2.2]{Sag20}).
What remains are the matrices $A$ whose
columns, all except for the last one, contain exactly one nonzero entry each,
and whose nonzero entries (apart from those in the last column)
are arranged from bottom to top as you move right in $A$. For example,
\[
A = \begin{pmatrix}
&  &  &  & p\\
&  &  & w & q\\
&  &  &  & r\\
& y & z &  & s\\
x &  &  &  & t
\end{pmatrix}
\]
is such a matrix (where empty cells mean zeroes, but the entries $q,s,t$ can also be zeroes). We call such matrices \emph{survivor matrices}. It is easy to see that such a survivor matrix $A \in \mathbb{N}^{I}_K$ is \textbf{entirely determined} by the knowledge of which of its rows end with a $0$.
(Indeed, this data determines which entries of $\comp(K)$ go into the last column of $A$ and which go into the other columns. The latter entries must then be arranged in distinct columns from left to right.)
If the $i$-th row of the matrix does not end with a zero, then it ends precisely with the $i$-th element of $\overline{K|I}$ (we remind the reader that $\overline{ K|I}$ is the sequence of the rightmost elements of each row of $A$). We can therefore encode this data as a nonempty subsequence $U$ of $\overline{K|I}$, which consists of those entries of $\overline{K|I}$ that are in the last column of $A$. Besides, there is exactly one column in $A$ for each element of $\comp(K)$ that is not in the last column and there are exactly $|\comp(K)| - |U| = |K|+1-|U|$ such elements. Add the last column to get:
\begin{equation}
\label{equation.colA}
\col(A) = 2 + |K| - |U|.
\end{equation}

\begin{example}
Let $\comp(K) = (1,2,3,1)$ and $\comp(I) = (3, 4)$ so that $\overline{ K|I} = (2,1)$. Then, the survivor matrices $A$ are
\[
A_1 = \begin{pmatrix}
0&0  &1  &2 \\
3&1  &0  & 0
\end{pmatrix},
\qquad
A_2 = \begin{pmatrix}
0&1  &2  &0 \\
3&0  &0  & 1
\end{pmatrix},
\qquad
A_3 = \begin{pmatrix}
0&1  &2 \\
3&0  & 1
\end{pmatrix}.
\]
Their respective corresponding subsequences $U$ are $U_1 = (2)$, $U_2 = (1)$ and $U_3 =  (2,1)$.
\end{example}

\noindent In Equation (\ref{equation.sum_A}) replace $\col(A)$ by the formula of Equation (\ref{equation.colA}) and reindex the summation over subsequences of $\overline{ K|I}$ to get Theorem \ref{thm.vnk}. 
\end{proof}

\subsection{Expansion of $\mathbf{V}_{n}^{(q)}\mathbf{B}_I$ in the $\mathbf{B}$-basis}

We use Theorem \ref{thm.vnk} to explicit the action of the $q$-deformed Dynkin operator on the $\mathbf{B}$-basis of the descent algebra $\Sigma_n$. 

\begin{theorem}
\label{thm.vnq}
Let $n$ be a positive integer and $q \in \mathbf{k}$ be an invertible parameter.
% This does not work for $n = 0$.
For $I, K \subseteq [n-1]$, the coefficient in $\mathbf{B}_K$ of $\mathbf{V}^{(q)}_n\mathbf{B}_I$ is $0$ if $I \nsubseteq K$ and otherwise given by
\begin{equation}
\label{equation.thm.vnq}
[\mathbf{B}_K]\mathbf{V}^{(q)}_n\mathbf{B}_I
= (-1)^{|K|} q^{n-1} (1-q^{-1})^{|I|} \prod_{v \in \overline{ K|I}}[v]_{q^{-1}}.
\end{equation}
Here, for any integer $v \in \NN$ and any $r \in \kk$, we define the $r$-integer $[v]_r$ to be $r^0 + r^1 + \cdots + r^{v-1}$. (Note that $[v]_{q^{-1}} = q^{-v+1} [v]_q$.)
\end{theorem}

\begin{proof}
We WLOG assume that $I \subseteq K$, since otherwise both sides are $0$ by the second sentence after Proposition~\ref{prop.solmac}.
Now, using \eqref{equation.vnqd} and Theorem \ref{thm.vnk}, we get
\begin{align*}
[\mathbf{B}_K]\mathbf{V}^{(q)}_n\mathbf{B}_I &= \sum_{k} (-q)^{k-1} [\mathbf{B}_K]\mathbf{V}_{n,k}\mathbf{B}_I
= \sum_{k} (-1)^{|K|+1}q^{k-1}\sum_{\substack{U \subseteq \overline{ K|I};\\ \Sigma U > n-k} }(-1)^{|U|}.
\end{align*}
Switching the summation order between $k$ and $U$, we rewrite this as
\begin{align*}
[\mathbf{B}_K]\mathbf{V}^{(q)}_n\mathbf{B}_I &=(-1)^{|K|+1} \sum_{U \subseteq \overline{ K|I}}(-1)^{|U|} \sum_{k \geq n - \Sigma U +1} q^{k-1}\\
&=(-1)^{|K|+1} \sum_{U \subseteq \overline{ K|I}}(-1)^{|U|} \sum_{k'=0}^{\Sigma U -1}q^{n-k'-1} \ \ \ \ \left(
 \text{reindex with }k'=n-k \right)  \\
%&=(-1)^{|K|+1}q^{n-1} \sum_{U \subseteq \overline{ K|I}}(-1)^{|U|} \sum_{k'=0}^{\Sigma U -1}q^{-k'}\\
&=(-1)^{|K|+1}q^{n-1} \sum_{U \subseteq \overline{ K|I}}(-1)^{|U|} \frac{1 - q^{-\Sigma U}}{1-q^{-1}}\\
&=(-1)^{|K|}\frac{q^{n-1}}{1-q^{-1}} \sum_{U \subseteq \overline{ K|I}}(-1)^{|U|}q^{-\Sigma U} ,
\end{align*}
where the last line uses the fact that $
\sum_{U \subseteq \overline{ K|I}}(-1)^{|U|}=0$ (since the sequence $\overline{ K|I}$ is nonempty). To end the proof, note that for any sequence $S$ and any parameter $x$,
\[
\prod_{s \in S}(1-x^s) = \sum_{T \subseteq S}(-1)^{|T|}x^{\Sigma T},
\]
where the summation is over all subsequences $T$ of $S$. As a consequence,
\begin{align*}
[\mathbf{B}_K]\mathbf{V}^{(q)}_n\mathbf{B}_I &=(-1)^{|K|}\frac{q^{n-1}}{1-q^{-1}} \prod_{v \in \overline{ K|I}}(1-q^{-v})
=(-1)^{|K|}q^{n-1}(1-q^{-1})^{|\overline{ K|I}|-1} \prod_{v \in \overline{ K|I}}{[v]}_{q^{-1}}.
\end{align*}
Finally, notice that $|\overline{ K|I}| = |I|+1$ to recover Theorem \ref{thm.vnq}.
\end{proof}

Theorem~\ref{thm.vnq} can also be proved via noncommutative symmetric functions, using \cite[Proposition 5.30, Theorem 4.17, Definition 5.1, Proposition 5.2]{NCSF2}.
Indeed, the image of $\VVq_n$ under the standard identification $\Sigma_n \to \operatorname{NSym}_n$ is the element $\Theta_n(q)$ from \cite[(66)]{NCSF2}, and thus differs by a $1-q$ factor from the $S_n((1-q)A)$ in \cite[Proposition 5.2]{NCSF2}.

%%%%%%%%%%%%%%%%%%%%%%%%%%%%%%%%%%
%%%%%%%%%%%%%%%%%%%%%%%%%%%%%%%%%%
%%%%%%%%%%%%%%%%%%%%%%%%%%%%%%%%%%
\section{Consequences of Theorem \ref{thm.vnq}}
%%%%%%%%%%%%%%%%%%%%%%%%%%%%%%%%%%
%%%%%%%%%%%%%%%%%%%%%%%%%%%%%%%%%%
%%%%%%%%%%%%%%%%%%%%%%%%%%%%%%%%%%
\subsection{Idempotence}
%%%%%%%%%%%%%%%%%%%%%%%%%%%%%%%%%%
Theorem \ref{thm.vnq} has various implications regarding the Dynkin operator and some of its generalisations. First, we get a new elementary proof that $\frac{1}{n}\mathbf{V}_n$ is an idempotent of $\mathbf{k}S_n$. 
\begin{corollary}
Equations \eqref{equation.vnbi}, \eqref{equation.vndi} and \eqref{equation.vn2vn} are direct consequences of Theorem \ref{thm.vnq}.
\end{corollary}
\begin{proof}
Setting $q=1$ in Equation \eqref{equation.thm.vnq}, we
\begin{vershort}
easily obtain
\end{vershort}
\begin{verlong}
obtain
\begin{equation*}
\mathbf{V}_n\mathbf{B}_I
= \begin{cases} \mathbf{V}_n, & \text{ if } I=\emptyset \\ 0 & \text{ otherwise}\end{cases}
\qquad \text{ for all } I \subseteq [n-1].
\end{equation*}
This yields
\end{verlong}
\eqref{equation.vnbi}.

Combining \eqref{equation.bidj} with \eqref{equation.vnbi}, we obtain
\begin{align*}
\VV_n \DD_I = \sum_{J \subseteq I}(-1)^{|I \setminus J|}\VV_n \BB_J
= (-1)^{|I \setminus \varnothing|}\VV_n \BB_\varnothing
= (-1)^{|I|}\VV_n,
\end{align*}
since $\BB_\varnothing = 1$.
Thus, \eqref{equation.vndi} is proved. Now, equations \eqref{equation.vnqd} (for $q=1$) and \eqref{equation.vndi} yield
\begin{align*}
\mathbf{V}_n\mathbf{V}_n &= \sum_{k=1}^n (-1)^{k-1} \VV_n \mathbf{D}_{[k-1]}
= \sum_{k=1}^n (-1)^{k-1} (-1)^{k-1} \VV_n
= \sum_{k=1}^n \VV_n
= n \VV_n .
\end{align*}
In other words, $\mathbf{V}_n^2 = n\mathbf{V}_n$, which proves \eqref{equation.vn2vn}.
\end{proof}
%
%%%%%%%%%%%%%%%%%%%%%%%%%%%%%%%%%%
\subsection{Zero coefficients and eigenvalues}
%%%%%%%%%%%%%%%%%%%%%%%%%%%%%%%%%%
\noindent Theorem \ref{thm.vnq} determines which of the $[\mathbf{B}_K]\mathbf{V}^{(q)}_n\mathbf{B}_I$ are zero for $I, K \subseteq [n-1]$.
\begin{corollary}
\label{corollary.zerocoeffs}
Let $\mathbf{k}$ be a field of characteristic $0$.
Fix $n \in \NN$ and $q \in \mathbf{k}$ and $I, K \subseteq [n-1]$.
If $q$ is a primitive $p$-th root of unity for $p>1$, then we have $[\mathbf{B}_K] \mathbf{V}^{(q)}_n\mathbf{B}_I = 0$ if and only if $I \nsubseteq K$ or $\overline{K|I}$ contains a multiple of $p$.
If $q = 1$, then $[\mathbf{B}_K] \mathbf{V}^{(q)}_n\mathbf{B}_I = 0$ if and only if $I \neq \varnothing$.
If $q = 0$, then $[\mathbf{B}_K] \mathbf{V}^{(q)}_n\mathbf{B}_I = 0$ if and only if $I \neq K$.
If $q$ is nonzero and not a root of unity, then $[\mathbf{B}_K] \mathbf{V}^{(q)}_n\mathbf{B}_I = 0$ if and only if $I \nsubseteq K$.
\end{corollary}

\begin{proof}
The case $q = 0$ follows from $\VV^{(0)}_n = \id$, whereas the case $q = 1$ follows from \eqref{equation.vnbi}.
The case $I \nsubseteq K$ is trivial by Theorem~\ref{thm.vnq}.
Now let $I \subseteq K$ and $q \neq 0, 1$.
According to \eqref{equation.thm.vnq}, we have $[\mathbf{B}_K]\mathbf{V}^{(q)}_n\mathbf{B}_I = 0$ iff some $v \in \overline{ K|I}$ satisfies $[v]_{q^{-1}}=0$.
But this happens precisely when
$q \neq 1$ and $q^v = 1$ for some $v \in \overline{ K|I}$.
%$q\neq1$ is a $p$-th root of unity with $p>1$ and at least one of the $v$ is a multiple of $p$. 
\end{proof}
We can now compute the eigenvalues of the action of $\mathbf{V}_n^{(q)}$ on $\Sigma_n$ by left multiplication -- i.e., of the linear endomorphism $u \mapsto \mathbf{V}_n^{(q)} u$ of $\Sigma_n$ --, which we name the eigenvalues of $\mathbf{V}_n^{(q)}$ for brevity. As the matrix entries $[\mathbf{B}_K]\mathbf{V}^{(q)}_n\mathbf{B}_I$ are in triangular form (being zero unless $I\subseteq K$), the eigenvalues of $\mathbf{V}_n^{(q)}$ are exactly its diagonal elements.
\begin{corollary}
\label{corollary.eigen}
Fix $n>0$ and $q\neq 1$. The eigenvalues of $\mathbf{V}_n^{(q)}$ are the elements of the family $(e_I)_{I \subseteq [n-1]}$ defined by
\begin{equation}
 (1-q)e_I = \prod_{v \in \comp(I)} (1-q^v).
\end{equation}
Note that if $q$ is a $p$-th root of unity with $p>1$ and $\comp(I)$ contains a multiple of $p$, then $e_I = 0$. 
\end{corollary}
\begin{proof}
Set $K=I$ in the formula of Theorem \ref{thm.vnq}.
\end{proof}

%%%%%%%%%%%%%%%%%%%%%%%%%%%%%%%%%%
\subsection{Image space dimension}
%%%%%%%%%%%%%%%%%%%%%%%%%%%%%%%%%%
We proceed with a corollary and some comments regarding the dimension of $\mathbf{V}_n^{(q)}\Sigma_n$. 
\begin{corollary}
\label{corollary.rank}
Let $n>0$. Let $\mathbf{k}$ be a field. The dimension of the subspace $\mathbf{V}_n^{(q)}\Sigma_n$ of the descent algebra $\Sigma_n$ is $2^{n-1}$ when $q$ is not a root of unity.
If $q$ is a primitive $p$-th root of unity with $p>1$,
\begin{equation}
\dim\left(\mathbf{V}_n^{(q)}\Sigma_n\right) = s_n^{(p)},
\end{equation}
where $s_n^{(p)}$ is the $n$-th Fibonacci number of order $p$ defined by $s_0^{(p)} = 0$ and $s_n^{(p)} = 2^{n-1}$ for all $1 \leq n < p$ and the recurrence relation
\begin{equation*}
s^{(p)}_n = 
%\begin{cases}
%0\; & \mbox{ if } \; n=0,\\
%2^{n-1}\; & \mbox{ if } \; 1\leq n < p,\\
s^{(p)}_{n-1} + s^{(p)}_{n-2} + \cdots + s^{(p)}_{n-p} %\; & \mbox{ if } \; 
\qquad \text{for all } n \geq p.
%\\
%\end{cases}
\end{equation*}
\end{corollary}
\begin{proof}
If $q$ is not a root of unity,
then all the eigenvalues $e_I$ of $\VVq_n$ are nonzero (see Corollary~\ref{corollary.eigen}),
and thus $\VVq_n$ is invertible, so that its image is the whole $\Sigma_n$. Next, assume that $q$ is a primitive $p$-th root of unity with $p>1$.
We let $\Comp_n^{(p)}$ denote the set of the compositions of $n$ that do not contain a multiple of $p$.
We have $|\Comp_n^{(p)}| = s_n^{(p)}$ easily\footnote{%
If $p >n$, then no composition of $n$ contains a multiple of $p$, and thus $|\Comp_n^{(p)}| = 2^{n-1}$.
When $n\geq p$, we split compositions in $\Comp_n^{(p)}$ according to their last entry $m$.
If $m <p$, we get a bijection with the compositions in $\Comp_{n-m}^{(p)}$ by removing the last entry.
If $m>p$, replace $m$ by $m-p$ to get a composition in $\Comp_{n-p}^{(p)}$.
This last construction is also bijective, and the desired recurrence relation follows.} and it remains to show that $\dim\left(\mathbf{V}_n^{(q)}\Sigma_n\right) = |\Comp_n^{(p)}|$.
To prove this, we denote
$
a_{IK} := [\mathbf{B}_K]\mathbf{V}^{(q)}_n\mathbf{B}_I  \text{ for all } I, K \subseteq [n-1] 
$
and argue in two steps.

%\begin{itemize}
%\item[(i)]
\textit{Proof of $\dim\left(\mathbf{V}_n^{(q)}\Sigma_n\right) \geq |\Comp_n^{(p)}|$:}
Recall that the operator $u \mapsto \mathbf{V}_n^{(q)} u$ on $\Sigma_n$ is triangular (the $a_{IK}$ are zero unless $I\subseteq K$). As a result, the dimension of $\mathbf{V}_n^{(q)}\Sigma_n$  is at least as large as the number of non-zero entries on the diagonal, i.e. the number of non-zero $e_I=a_{II}=[\mathbf{B}_I]\mathbf{V}^{(q)}_n\mathbf{B}_I$ when $I$ ranges over all subsets of $[n-1]$.
%According to Corollary \ref{corollary.eigen}, all these $e_I$ are non-zero if $q$ is not a root of unity.
Since $q$ is a primitive $p$-th root of unity, Corollary \ref{corollary.eigen} shows that we have $e_I =0$ iff $\comp(I)$ contains at least one multiple of $p$, that is, iff $\comp(I) \notin \Comp_n^{(p)}$.
As a result, $\dim\left(\mathbf{V}_n^{(q)}\Sigma_n\right) \geq |\Comp_n^{(p)}|$.

%\item[(ii)] 
\textit{Proof of $\dim\left(\mathbf{V}_n^{(q)}\Sigma_n\right) \leq |\Comp_n^{(p)}|$:}
The next step is to show that $\dim\left(\mathbf{V}_n^{(q)}\Sigma_n\right) \leq |\Comp_n^{(p)}|$.
By the rank-nullity theorem, it will suffice to find $|\Comp_n \setminus \Comp_n^{(p)}|$ many linearly independent linear relations between the coefficients of each element of $\mathbf{V}_n^{(q)}\Sigma_n$.
Let $K \subseteq [n-1]$ be such that $\comp(K) = (\alpha_1, \alpha_2, \dots, \alpha_{|K|+1}) \in \Comp_n \setminus \Comp_n^{(p)}$.
Thus, some $\alpha_i$ is a multiple of $p$.
Let $m$ be maximal such that $\alpha_m$ is a multiple of $p$. 
\begin{itemize}
\item[(a)] If $m=|K|+1$ then
\begin{align}
\label{eq.linrel0q.1}
\forall f \in \mathbf{V}_n^{(q)}\Sigma_n, \ \ [\BB_K]f = 0.
\end{align}
Indeed, by linearity, it suffices to show that $a_{IK} = 0$ for each $I \subseteq [n-1]$. Either $I \nsubseteq K$ and $a_{IK} = 0$, or $I \subseteq K$ and (by construction) $\overline{K|I}$ contains $\alpha_{|K|+1}$ and thus $a_{IK}$ contains the factor $[\alpha_{|K|+1}]_{q^{-1}} = 0$ (see Theorem \ref{thm.vnq}). 
\item[(b)] Next assume $m < |K|+1$.
%The next part $\alpha_{m+1}$ is not a multiple of $p$.
Then, define the set $K' \subseteq [n-1]$ such that 
\[
\comp(K') = (\alpha_1, \dots, \alpha_{m-1}, \alpha_m + \alpha_{m+1}, \alpha_{m+2},\dots, \alpha_{|K|+1}).
\]
As a result $K' \subset K$. The composition $\comp(K')$ has exactly one part less than $\comp(K)$.
We fix any $I \subseteq [n-1]$, and aim to show that $a_{IK'} = - a_{IK}$.
Indeed, if $I \nsubseteq K$ and $I \nsubseteq K'$, then $a_{IK}=a_{IK'}=0$. If $I \subseteq K$ but $I \nsubseteq K'$, then $\alpha_m \in \overline{K|I}$ and $a_{IK} = 0$ (having a factor of $[\alpha_m]_{q^{-1}}$). But $a_{IK'} =0$ as well, since $I \nsubseteq K'$.
Finally, if $I \subseteq K' \subseteq K$, then either $\alpha_{m+1} \notin \overline{K|I}$ and the only difference between $a_{IK}$ and $a_{IK'}$ is the coefficient $(-1)^{|K|} = - (-1)^{|K'|}$; or $\alpha_{m+1} \in \overline{K|I}$ and an additional difference between coefficients $a_{IK}$ and $a_{IK'}$ is that the factor $[\alpha_{m+1}]_{q^{-1}}$ in $a_{IK}$ is replaced by $[\alpha_{m}+\alpha_{m+1}]_{q^{-1}}$ in $a_{IK'}$. But as $q$ is a $p$-th root of unity and $\alpha_{m}$ is a multiple of $p$,
we have $[\alpha_{m}+\alpha_{m+1}]_{q^{-1}} = [\alpha_{m+1}]_{q^{-1}}$
and therefore $a_{IK'} = - a_{IK}$ for all $I$. By linearity, we conclude that
\begin{align}
\label{eq.linrel0q.2}
\forall f \in \mathbf{V}_n^{(q)}\Sigma_n, \ [\BB_{K'}]f = -[\BB_K]f.
\end{align}
\end{itemize}
Thus, for each $K \subseteq [n-1]$ with $\comp(K) \in \Comp_n \setminus \Comp_n^{(p)}$, we have found a linear relation -- either \eqref{eq.linrel0q.1} or \eqref{eq.linrel0q.2} -- that holds for the coefficients of all $f \in \mathbf{V}_n^{(q)}\Sigma_n$.
These relations are linearly independent due to distinct ``leading terms''.
Hence, $\dim\left(\mathbf{V}_n^{(q)}\Sigma_n\right) \leq |\Comp_n^{(p)}|$ is proved,
so that $\dim\left(\mathbf{V}_n^{(q)}\Sigma_n\right) = |\Comp_n^{(p)}|$ follows.
%\qedhere
%\end{itemize}
\end{proof}

Corollary \ref{corollary.rank} shows that the image $\mathbf{V}^{(q)}_n\Sigma_n$ of the left action of $\mathbf{V}^{(q)}_n$ on $\Sigma_n$ is a proper subspace of $\Sigma_n$ when $q$ is a root of unity of order $\leq n$.
%and kernel defined by $X, \ \mathbf{V}^{(q)}_n X = 0$.
This reveals several connections:

\textbf{1.} In the special case $p=2$ i.e. $q=-1$, the basis of the image of $\mathbf{V}^{+}_n:=\mathbf{V}^{(-1)}_n$ is indexed by \emph{odd compositions} of $n$ (compositions with only odd entries). This reminds one of the \emph{peak algebra} $P_n$ of order $n$, a left ideal of $\Sigma_n$ where permutations with the same \emph{peak set} have the same coefficient. Furthermore $\mathbf{V}^{+}_n$ is used in \cite{Sch05} to provide an analogue of the Dynkin operator in the peak algebra. Using Theorem \ref{thm.vnq}, we provide an explicit expansion of $\mathbf{V}^{+}_n\mathbf{B}_I$, namely:
\begin{equation}
\mathbf{V}^{+}_n\mathbf{B}_I = 2^{|I|}\sum_{\substack{K \supseteq I; \\ \overline{K|I} \cap 2\NN = \emptyset}}(-1)^{n-1-|K|}\mathbf{B}_K.
\end{equation}

\textbf{2.} When $q$ is a primitive $p$-th root of unity with $p>1$, the dimension of the image of $\mathbf{V}^{(q)}_n$ is equal to the dimension of the vector space of \emph{$(p-1)$-extended peaks quasisymmetric functions of degree $n$} that we introduced in \cite{GriVas23}.
%Summing over all $n$ we get the \emph{algebra of $(p-1)$-extended peaks}, a subalgebra of the ring of quasisymmetric functions $\QSym$ that interpolates between Stembridge's peak algebra of quasisymmetric functions \cite{Ste97} and $\QSym$ itself. 

\begin{remark}
In \cite{CalderbankHanlonSundaram94}, it is claimed that the action of $\mathbf{V}_n^{(q)}$ on $\Sigma_n$ is diagonalizable for all $q$.
However, this is false for $n = 7$ and $q = e^{\pi i/3}$ (there is a nontrivial Jordan block for eigenvalue $1$).
This counterexample was found by GPT-5.5 and independently verified using SageMath.

The alleged proof of diagonalizability in \cite{CalderbankHanlonSundaram94} relies on the semisimplicity claim in \cite[Corollary 2.3]{CalderbankHanlonSundaram94}, which is false in general.
The diagonalizability of $\mathbf{V}_n^{(q)}$ does hold for generic $q$, as was shown in \cite[Theorem 5.14]{NCSF2}.

Perhaps the $0$-eigenvalue of $\mathbf{V}_n^{(q)}$ might be semisimple (i.e., all Jordan blocks for eigenvalue $0$ are of size $1$)?
\end{remark}

%In the next section, we look at a few extensions of our work to the peak algebra and peak Dynkin operator. 
%%%%%%%%%%%%%%%%%%%%%%%%%%%%%%%%%%%
%%%%%%%%%%%%%%%%%%%%%%%%%%%%%%%%%%%
%%%%%%%%%%%%%%%%%%%%%%%%%%%%%%%%%%%
%\section{The peak Dynkin operator and the peak subalgebra}
%%%%%%%%%%%%%%%%%%%%%%%%%%%%%%%%%%%
%%%%%%%%%%%%%%%%%%%%%%%%%%%%%%%%%%%
%%%%%%%%%%%%%%%%%%%%%%%%%%%%%%%%%%%
%\begin{equation}
%V_n\Pi_I = \begin{cases}(-1)^{|I|}V_n \mbox{ if $n$ is odd and $I \subset 2\mathbb{N}$} \\ 0 \mbox{ else}\end{cases}
%\end{equation}
%Denote $R_I = \sum_{J \subseteq I} \Pi_J$ then
%\begin{equation}
%V_n R_I = \begin{cases}0 \mbox{ if $n$ is even or $I \bigcap 2\mathbb{N} \neq \emptyset $} \\ V_n \mbox{ else} \end{cases}
%\end{equation}

\acknowledgements{We thank the referee for highly useful comments and references.}

%% if you use biblatex then this generates the bibliography
%% if you use some other method then remove this and do it your own way
\printbibliography

\end{document}